%% file: global_CL-heuristic.tex
\documentclass{compositio}
\usepackage{etex}
\usepackage{amsmath}
\usepackage{global_CL}
\usepackage{dsfont}

\begin{document}

\title{The Global Cohen-Lenstra Heuristic}
\author{Johannes Lengler}
\email{johannes.lengler@math.uni-sb.de}
\address{Mathematisches Institut\\
Universitaet des Saarlandes\\
Saarbruecken, Germany}
%
%

\begin{abstract}
The Cohen-Lenstra heuristic is a universal principle that assigns to each group a probability that tells how often this group should occur ``in nature''. The most important, but not the only, applications are sequences of class groups, which behave like random sequences of groups with respect to the so-called Cohen-Lenstra probability measure. 

So far, it was only possible to define this probability measure for finite abelian $p$-groups. We prove that it is also possible to define an analogous probability measure on the set of \emph{all} finite abelian groups when restricting to the $\Sigma$-algebra on the set of all finite abelian groups that is generated by uniform properties, thereby solving a problem that was open since 1984.

\end{abstract}

\maketitle


\section{Introduction}
\label{sec:introduction}

\input{introduction}

\input{global_CL_body}

\begin{acknowledgements}
This paper has evolved from my PhD-thesis. I want to thank my thesis advisor, Ernst-Ulrich Gekeler, for his guidance and his aid.
\end{acknowledgements}


\bibliographystyle{amsalpha}
\bibliography{global_CL-heuristic}

\end{document}

%% file: introduction.tex

In the last decades, a method has gained ever-increasing influence which treats deterministic objects as if they were random objects and studies them with probability theoretic means. A major breakthrough for this method came in 1984, when Henri Cohen and Hendrik W. Lenstra noticed that the sequence of class groups of quadratic number fields seems to behave essentially like a random sequence with respect to a certain probability distribution on the space of all finite abelian groups. 

The distribution is determined by the requirement that the measure of a group should be inversely proportional to the size of its automorphism group.

Later on, it turned out that this distribution occurs also in many other contexts and plays the role of a ``natural'' distribution, regulating the structure of finite abelian groups in all situations where no obvious structural obstacles for a random-like behaviour exist. 

The consequences of such random behaviour are immense. E.g., it implies that each group appears with a positive density. Cohen and Lenstra conjectured this for the class groups of imaginary quadratic number fields (excluding the $2$-part). The conjecture is still unproven, and in fact, it is even unknown whether there are groups that appear infinitely often as such class groups.\bigskip

Unfortunately, so far it was only possible to define probability distributions for ``local'' groups, i.e., for $p$-groups where $p$ is some fixed prime. In a global setting, researchers needed to restrict themselves to contents, with a variety of problems arising (discussed in section \ref{sect:Doingnothing}).

In this paper we show that it is possible to define a global probability measure which is compatible with the Cohen-Lenstra heuristic by restricting the measure to a (still rich) set of measurable sets of finite abelian groups. The $\Sigma$-algebra of measurable sets is generated by sets defined via ``uniform properties'' (definition \ref{def:localprop}). The measure is given by explicit formulas and may be easily computed. In my eyes, this is a very satisfactory answer to the problems that were latently present in the research community for the last 25 years.\bigskip

The structure of the paper is as follows: The remainder of this section provides preliminaries and notation, including an introduction to the local Cohen-Lenstra probability measure. Furthermore, I will explain which problems arise when we try to transfer the local heuristic to the global case. In section 2, we will study the global theory. First, we will revise the classical approaches using densities and contents, and their drawbacks. I will make a proposal for a general notion of a Cohen-Lenstra content, independent of the specific application.

In the second part of section 2, we will introduce the notion of uniform properties (definition \ref{def:localprop}), define the global Cohen-Lenstra measure (definition \ref{def:globalmeasure}) and state the main theorem that this yields indeed a measure (theorem \ref{thm:globalmeasure}). The proof of the main theorem is complicated -- all of section 3 is devoted to it. In section 4, we study some extensions and variations of the global Cohen-Lenstra measure. Finally, in section 5 we show that if we combine the global measure with the classical approaches then it loses the measure properties, but stays restrictedly countably additive.



\subsection{Preliminaries and notation}

By $\PP$, I denote the set of all integer primes. 

Throughout the paper, I will only consider finite abelian groups. For brevity, we will write ``group'' to mean ``finite abelian group'', and ``$p$-group'' to mean ``finite abelian $p$-primary group'', i.e., a group with order a power of $p$. Furthermore, we will consider groups only up to ismomorphism, so a phrase like ``sum over all finite abelian groups'' really means that the sum runs over all isomorphism classes of finite abelian groups.

$\GG$ is the set of all (isomorphism classes of) finite abelian groups, and $\GG_p$ is the set of all (isomorphism classes of) finite abelian $p$-groups, for a prime $p$. 

For a finite set $M$, I will denote its cardinality by $\#M$.\medskip

For a finite abelian group $G$, we write $\Aut(G)$ for its automorphism group. The \emph{order $\ord(G)$} is the number of elements of $G$, the \emph{rank $\rk(G)$} is the minimal number of generators. The \emph{exponent $\exp(G)$} is the minimal integer $n>0$ such that $n\cdot G = \{0\}$.\medskip

A finite abelian $p$-group can be uniquely (up to isomorphism) written in the form

$$\prod_{i=1}^k (\ZZ/p^{e_i})^{r_i},$$
where $k\in \NN$, $e_i,r_i \in \NN^+$ for all $i$, and where $e_1>e_2>\ldots > e_k$.

So the quantities $k$, $e_i$, and $r_i$ determine the group. A collection of such quantities is called an \emph{integer partition}, and we call the set of all partitions $\GG_P$. Each partition corresponds to a way to write an integer $n$ as a sum of positive integers, up to order of summation (see \cite{Len09} or \cite{AE04} for details). By definition, for any prime $p$ we have a canonical bijection $\GG_{P} \stackrel{\cong}{\rightarrow} \GG_{p}$. 

We will need one more lemma about the number of partitions of a given size (and hence, the number of $p$-groups of a given order):

\begin{lemma}\label{lem:partgrowth}
The number $a(n)$ of partitions of $n$ satisfies $a(n) \leq F_{n+1}$, where $F_k$ denotes the $k$-th Fibonacci number, given by $F_1:=F_2:=1$ and $F_{n+1} := F_n+F_{n-1}$.

In particular, $a(n) \in O(\phi^n)$, where $\phi = \frac{1+\sqrt{5}}{2}$ is the golden ratio.
\end{lemma}

\begin{proof}
\cite[3.3]{AE04}.
\end{proof}


\subsection{The local Cohen-Lenstra heuristic}

Let me first give a probabilistic formulation of the Cohen-Lenstra heuristic:

\emph{Let $p$ be a prime. Assume we have a ``natural'', unbiased stochastic process producing finite abelian $p$-groups. If we fix a finite abelian $p$-group $G$ then the probability that an output of the process is isomorphic to $G$ is inversely proportional to the size of its automorphism group $\Aut(G)$.}\index{Cohen-Lenstra heuristic!local}

In this formulation, the heuristic is not a theorem but rather a meta-principle. It first became popular by the famous paper \cite{CL83} of Henri Cohen and Hendrik W. Lenstra. In honor to this paper I call the principle ``Cohen-Lenstra heuristic'' or ``Cohen-Lenstra principle''. In their paper they claimed (without proof, but with some evidence) that the sequence of $p$-parts of class groups of imaginary quadratic number fields (which is a deterministic sequence!) behaves essentially like a random sequence in the above sense, for $p\neq 2$.

In the definition above, ``unbiased'' is not a precise term but rather means that we do not allow obvious obstacles. For example, there might well be stochastic processes that produce only cyclic groups. Or some that produce only groups of rank at most $2$ (as is the case for the point group of elliptic curves over various finite fields). Such processes may well be modelled via a probabilistic approach (done so for the elliptic curves in \cite{Gek06}), but the probability distribution is clearly biased.

The sequence of class groups of number fields is the most famous application of the Cohen-Lenstra heuristic --- not only for quadratic extensions of $\QQ$, but also much more general number field extensions are seemingly governed by similar heuristics, which may be derived from the Cohen-Lenstra heuristic. Note that apart from some special cases, all statements are conjectural but are supported by strong numerical and theoretical evidence. You may consult \cite{Mal06} or \cite{Len09} for details. 

However, class groups of number fields are by far not the only application. The conjectures may be transferred to function field extensions, and in this setting, Pacelli \cite{Pac04}, Gekeler \cite{Gek06}, Achter \cite{Ach06}, and other researchers could prove some of the conjectures --- see \cite{Ach06} for an overview. Furthermore, there are completely different applications such as generating a (finite abelian) $p$-group by choosing generators and imposing random relations on them with respect to some canonical Haar measure --- or such as the size of conjugacy classes of the general linear group $\GL(n,p)$ (cf. \cite{Len09}).\bigskip

Returning to the above formulation of the Cohen-Lenstra heuristic, you may have noticed that it refers only to $p$-groups, not to general (finite abelian) groups. This is due to the following reason: Let $w_p$ be the measure on the set $\GG_p$ which is defined via

$$w_p(\{G\}) = \frac{1}{\#\Aut(G)}\qquad \text{for all one-element sets $\{G\}\subset \GG_p$}.$$ 

We write for short $w(G)$ instead of $w(\{G\})$.

Then $w_p$ is well-defined because $\GG_p$ is countable. We call $w_p$ the (local) \emph{Cohen-Lenstra weight}. However, we want to scale $w_p$ into a probability measure. This is only possible if the total measure of $\GG_p$ is finite. Fortunately, this is the case. Cohen and Lenstra have proven \cite{CL83} that

$$w_p(\GG_p) = \prod_{i=1}^{\infty} (1-p^{-i})^{-1} < \infty.$$

Our heuristic tells us that the probability for each group $G$ is supposed to be proportional to $\#\Aut(G)^{-1}$, so it must be 

$$P_p(G) = \frac{w_p(G)}{w_p(\GG_p)} = \frac{1}{\#\Aut(G)}\prod_{i=1}^{\infty} (1-p^{-i}),$$

which induces a well-defined probability measure on $\GG_p$.\bigskip

Now that we have understood the local setting, we may turn to non-primary groups.

\subsection{The global Cohen-Lenstra heuristic}

We have seen how the Cohen-Lenstra principle leads to a probability distribution on the set of all finite abelian $p$-groups, for arbitrary $p\in \PP$. However, being a $p$-group is a restriction we would like to remove. Often we deal with non-primary groups, e.g., the class group of a number field or the Jacobian of a hyperelliptic curve.

But when we try to transfer the above process for $p$-groups to non-primary groups, we find that $w(\GG) =\infty$, so our approach fails to give us a probability measure. Moreover, we cannot transfer the local heuristic to the global case. In fact, \emph{for general finite abelian groups, there is no probability distribution that would allow us to perform a stochastic process generating random sequences of groups that is compatible with the Cohen-Lenstra heuristic}, as we can do when restricting to $p$-groups. So we can not compare a sequence of groups with a random sequence because there is no adequate stochastic process that could generate such a random sequence.

%
%

In this paper we will discuss a way still to assign probabilities to certain events. We imitates the definition of the Lebesgue measure on $\RR^n$. In real analysis, constructions like the Banach-Tarski-paradox (originally in \cite{BT24}; for a more recent treatment see \cite{Wag93}) show that there is no equivariant measure on the power set of $\RR^n$. The solution is to designate only some $\sigma$-algebra of sets as \emph{measurable} and to define the measure only on those. We copy this approach by defining \emph{uniform properties} and designating these as a basis for the $\sigma$-algebra. In my eyes, this is a quite satisfactory solution.



This is the first time that a probability measure is introduced in the context of the global Cohen-Lenstra heuristic. In the literature, people used densities instead. However, this approach has severe theoretical and practical drawbacks, which will be discussed in section \ref{sect:Doingnothing}.\bigskip

Let me add some remarks about terminology: Cohen and Lenstra speak of probabilities, although they are only talking about contents (cf. def. \ref{def:content} below), and they are well aware of this terminological slackness. I will not use the term ``probability'' in a context where we do not have a probability measure -- therefore, my terminology is different from the one of Cohen and Lenstra. When I talk about their concept, I use the words ``content'' or ``density''.  Further, I use the word ``heuristic'' to refer to any one of the above concepts, so a ``heuristic'' is not a precise mathematical concept.

It would be very convenient to write down a definition of ``the'' Cohen-Lenstra content. Unfortunately, such a definition does not exist. (This is one of the circumstances that necessitate this chapter!) Rather, the precise definitions in the literature (which still include imprecise terms like ``reasonable functions'') work with the concept of densities (section \ref{sect:Doingnothing}) and always depend on the specific application. For different applications, one gets different densities: they differ in the set of ``measurable'' sets, but even if one set is assigned a content in several settings, these contents need not agree. These problems are discussed in more detail in section \ref{sect:Doingnothing}. I will define a global content in definition \ref{def:globalCLmeasure1} as my personal proposal of a theoretically sound content, but you should be aware that in the literature there is no agreement on what a ``Cohen-Lenstra content'' should be (at least if you want it to be independent of the specific application). 

Opposed to that, when I talk about ``the global Cohen-Lenstra measure'' or about ``the global Cohen-Lenstra probability'', I mean the probability measure that I define in \ref{def:globalmeasure}. Its existence is the central insight of this paper, and section \ref{sect:globalmeasure} is devoted to studying this measure.

%% file: global_CL_body.tex
\section{Global theory}\label{chap:globaltheory}

\subsection{Global contents}\label{sect:content}

Before we start, let me repeat some basic notions from measure theory.

\begin{definition}\label{def:content}
An \emph{algebra} of sets over some set $X$ is a set $\AAA$ of subsets of $X$ that is closed under complements, finite unions and finite intersections, and with $\emptyset \in \AAA$.

A \emph{$\sigma$-algebra} is an algebra that is also closed under countable unions and intersections. We usually denote $\sigma$-algebras by $\Sigma$.

A \emph{content} on an algebra $\AAA$ is a map $\mu: \AAA \to \RR\cup\{\infty\}$ such that
\begin{itemize}
\item $\mu(\emptyset) = 0$.
\item $\mu(A) \geq 0$ for all $A \in \AAA$.
\item $\mu(A_1 \cup A_2) = \mu(A_1) + \mu(A_2)$ for all disjoint $A_1,A_2\in \AAA$.
\end{itemize}

We will usually further assume that $\mu(X) = 1$.

A content that is defined on a $\sigma$-algebra is called a \emph{measure} if it is furthermore countably additive. If $\mu(X) = 1$, it is called a \emph{probability measure}.
\end{definition}

\begin{remark}
In the literature, contents are more often referred to as \emph{finite additive measures}. I have not adopted this notion because it suggests that finite additive measure are measures, which is not true in general.
\end{remark}

Before coming to the different methods of defining contents, let me first illustrate the problems we face when we try to define a global probability measure. So assume we are more ambitious and want to construct a measure instead of a content.

What properties should a global probability measure have? Note that for any $p$, there is a natural projection $\GG \stackrel{\pi_p}{\to} \GG_p$. We would like our probability to be compatible with these maps, i.e., for any $M \subseteq \GG_p$ we would like to have $P(\pi_p^{-1}(M)) = P_p(M)$. ($P_p$ is the local Cohen-Lenstra probability on $\GG_p$.)
%
%

Moreover, the $p$-parts of each group should be independent (as the automorphism group of a group decomposes into a direct product of the automorphism groups of its $p$-parts), i.e., for finitely many mutually distinct primes $p_1, \ldots, p_k$ and sets $M_i \subseteq \GG_{p_i}$, $1\leq i \leq k$ we require 
\begin{equation}
P\left(\bigcap_{i}\pi_{p_i}^{-1}(M_i)\right)  = \prod_iP_{p_i}(M_i).\label{eq:independence}
\end{equation}

So the first attempt would be to define $\Sigma$ as the coarsest $\sigma$-algebra that contains all $\pi_p^{-1}(M)$ for all primes $p$ and $M\subseteq \GG_p$, and to define the probabilities via the product formula.\label{plain:measurelocalstuff}

Unfortunately, this does not lead to a measure: Obviously we can describe every group $G \in \GG$ by specifying each of its $p$-parts. Since a measure is defined on a $\sigma$-algebra, the set $\{G\}$ would be measurable as a countable intersection of measurable sets, and by an easy calculation it would have measure $0$. But since $\GG$ is countable, we would get the contradiction 
$$1 = P(\GG) = P\left(\bigcup_{G\in \GG} \{G\}\right) = \sum_{G\in \GG}P(\{G\}) = 0.$$

Note that this argument shows that $\Sigma$ is the whole power set of $\GG$.

%

We see that it is difficult to find a measure which is compatible with the local Cohen-Lenstra measures, although we will finally succeed in section \ref{sect:globalmeasure}. Before I come to this measure, let us discuss the alternatives. In the following sections, I illustrate several ways of defining contents instead of measures. However, we will also find that all these alternatives have severe drawbacks.

\subsubsection{Densities}\label{sect:Doingnothing}

Cohen and Lenstra tried to avoid the problems illustrated above in the following way: They were interested in a very concrete sequence of finite abelian groups (the sequence of non-$2$-parts of class groups of imaginary quadratic number fields). For us, the concrete sequence is of no importance, so let $(G_n)$ be a sequence of finite abelian groups. Let $D$ be the set of all subsets $S\subseteq \GG$ which have a density in $(G_n)$, i.e., all $S$ for which the limit\index{density}

$$\lim_{n\to \infty} \frac{\#\{k \leq n \mid G_k \in S\}}{n}$$

exists. Then $D$ is an algebra of sets, and the limits define a content on $D$. This approach is copied by almost all currently active researchers. It has the philosophical drawback that we cannot speak of probabilities, and the practical drawback that we usually do not know $D$. Furthermore, it is at least annoying that we do not have countable additivity. But there are also much more severe obstacles.

Of course, we want to decide whether a sequence is compatible with the (local) Cohen-Lenstra distributions. But how do we decide this? In principle we would like $D$ to be ``reasonably'' rich, and that the densities of sets $S \in D$ are compatible with the Cohen-Lenstra heuristic.

But what does ``compatible'' really mean? Often, researchers are only concerned with very special sets $S$, in particular sets that are direct products $\prod_{p\in\PP}S_p$, for sets $S_p \in \GG_p$. Then they declare the Cohen-Lenstra probability to be $\prod_{p\in\PP}P_p(S_p)$. This sounds quite reasonable, but in this way there is no hope whatsoever to gain countable additivity, as is proven in section \ref{sect:globalquant}.

If we are given such a direct product set $S$, are there other ways to define a ``Cohen-Lenstra probability'' for $S$? The answer is yes! We have two different limit processes going on: One in the definition of the local Cohen-Lenstra probability, where we average over all $p$-groups. And another one when we multiply the probabilities for various primes. Assigning the probabilities $\prod_{p\in\PP}P_p(S_p)$ to a set $S$ as above imposes an order on the limit process. Moreover, by what we have already shown, the limits do not commute! So we might with equal legitimation compute the double limit in a different way, and obtain a different ``Cohen-Lenstra probability'' for the same set $S$. This is highly unsatisfactory.

Another point is that for every sequence $(G_n)$ we get a different content. Even if we would accept the order of the limit process for special sets $S= \prod_{p\in\PP}S_p$, then it is not clear at all how to extend this to the whole power set of $\GG$. For a set $S$ which does not happen to be a direct product, there are many ways that lead to different contents for $S$, and we do not have a canonical way of choosing the ``right'' one. Thus for each sequence of groups, we would have to figure out the sets with densities and make up a new content on these sets. For different sequences of groups, the contents would in general not be compatible.

A related approach, which appears to be a bit less critical, is to define a content $P(S)$ for any set $S\subseteq \GG$ for which the following limit exists:

$$P(S):=\lim_{x\to\infty}\frac{\sum_{G\in S,|G|<x} w(G)}{\sum_{|G|<x} w(G)},$$

where $w(G) = \frac{1}{\#\Aut(G)}$ is the Cohen-Lenstra weight of $G$.

This yields a content. Basically, the approach imposes an ordering onto $\GG$, namely by their size, and then sums up over all groups up to a certain threshold. This sounds very natural, but still it is a specific ordering. It corresponds to taking the density with respect to the sequence where the group of order $1$ appears an appropriate number of times, then the group of order $2$ appears, and so on. This analogy is not perfect, because it is only possible to construct the sequence for every finite start sequence $\{G\in\GG \mid \ord(G) \leq x\}$ of the ordering.\footnote{In order to extend the sequence, we need to adjust the number of order-1-groups, order-2-groups, \ldots\ in order to get an integral number of appearances. E.g., in order to simulate the weights $(1)$, $(1,\frac{1}{2})$, and $(1,\frac{1}{2},\frac{1}{3})$, we must take the sequences $(G_1)$, $(G_1,G_1,G_2)$, and $(G_1,G_1,G_1,G_1,G_1,G_1,G_2,G_2,G_2,G_3,G_3)$, respectively. So we do not simply add groups
 to the tail of the sequence when considering additional weights.} Nevertheless, in my eyes the analogy catches the essential point: There is no real reason to impose this specific ordering on $\GG$, and it is not clear why a truly random sequence should respect this specific ordering.

Furthermore, is the ordering above  really the most natural ordering? Or would it perhaps be more natural to order the groups by their weight? This would give a different content, and so it would be a matter of taste which content one prefers. We see that this situation is quite unsatisfactory. 

Finally, by this approach, we do not have any hope to get a measure. Clearly, every one-element set $S=\{G\}$ is measurable with measure $0$, which already rules out countable additivity. 

Summarizing, the illustrated approaches only postpone the problems -- the reason why they have worked so far is that only a very limited type of sets $S$ has been investigated, and that often the researcher concentrates on only one specific sequence of groups and does not care about other sequences. Cohen and Lenstra were well aware of the problem (that is why they did not specify what a ``reasonable function'' \cite[8.1]{CL83} should be), but apparently they saw no way to avoid it.


We have seen that we need a general notion for sequences of groups to be ``compatible'' with the Cohen-Lenstra heuristic. Let us first try to define a content that does not depend on the specific approach. In order to be compatible with the local Cohen-Lenstra measures, we want the algebra of sets to contain all sets of the form $\pi_p^{-1}(M)$, where $p\in\PP$ and $M\subseteq\GG_p$. This leads to the following definition:

\begin{definition}\label{def:globalCLmeasure1}
Let $\AAA$ be the algebra of all subsets $S$ of $\GG$ for which there exists a finite index set $I \subset \PP$ and a set $S_{I}\subseteq \prod_{p\in I} \GG_p$ such that\footnote{the symbol ``$\bigoplus$'' denotes the outer direct sum, by which I simply mean for any index set $I \subseteq \PP$:
$$\bigoplus_{p\in I} \GG_p := \{(G_p)_{p\in I} \in \prod_{p\in I} \GG_p \mid \text{ almost all $G_p$ are $0$}\}.$$
So in particular, 
$$\bigoplus_{p\in \PP} \GG_p \stackrel{\cong}{\longrightarrow} \GG.$$} 
\begin{equation}\label{eq:topology}
S = S_{I} \times \bigoplus_{p\in \PP \setminus I} \GG_p.
\end{equation}

Informally speaking, $S$ is only specified at finitely many local places.

We define the \emph{(global) Cohen-Lenstra content $P$}\index{Cohen-Lenstra content!global} on $\AAA$ via 

$$P(S_{I} \times \bigoplus_{p\in \PP \setminus I} \GG_p) := \sum_{G\in S_I}\prod_{p\in I}P_p(G_p),$$
where $G_p$ denotes the $p$-part of $G$.

I usually omit the attribute ``global'' if no confusion is possible and talk only of the \emph{Cohen-Lenstra content} on $\GG$.
\end{definition}

%
%
%

\begin{theorem}\label{globalCLmeasure}
The global Cohen-Lenstra content is a well-defined content.
\end{theorem}

\begin{proof}
It is clear that $\AAA$ is an algebra of sets.

By measure theory we know that for any finite $I$ we can endow $\prod_{p\in I} \GG_p$ with a probability measure by defining $P(\{G\}) := \prod_{p\in I}P_p(G_p)$. (Note that this does not work for infinite $I$ because $\GG$ is not the product space but rather the direct sum of the $\GG_p$ -- only for finite $I$ do $\prod_{p\in I} \GG_p$ and $\bigoplus_{p\in I} \GG_p$ agree.) Since any complement and any finite union or finite intersection of sets in $\AAA$ is only specified on a finite set $S$, we can restrict ourselves to a probability space of this kind. So we may restrict ourselves to the power set of $\prod_{p\in I} \GG_p$, where $I$ is some \emph{finite} set of primes. But the finite product of probability spaces is again a probability space, so all formulas then become evident.
\end{proof}

%
%

\subsubsection{Global quantities}\label{sect:globalquant}

What kind of statements would we like to make about groups? We have already seen that we cannot measure all sets of groups. But there are some minimal requirements -- at least to my feeling we should be able to measure the three most important quantities of a finite abelian group: its order, rank and exponent. So we would like the following sets to be measurable for any $n$:

\begin{itemize}
\item $\{G \in \GG \mid \ord(G) = n\}$.
\item $\{G \in \GG \mid \rk(G) = n\}$.
\item $\{G \in \GG \mid \exp(G) = n\}$.
\end{itemize}

Unfortunately, it is impossible to achieve this with a measure. We will prove:

\begin{theorem}\label{thm:globalorderandexponent}
Let $\AAA$ be any algebra on $\GG$ with a content $P$ that is compatible with the Cohen-Lenstra heuristic induced by the projections $\GG \to \GG_p$. If the order or the exponent is measurable, then there exist countably many measurable, pairwise disjoint sets of measure $0$ and with union $\GG$.
\end{theorem}

Note that we implicitly assume that distinct primes are independent of each other (in the sense of equation (\ref{eq:independence}) on page \pageref{eq:independence}). This is an assumption which is usually made whenever people work with the Cohen-Lenstra philosophy.

\begin{proof}
We will only show the statement for the order. The statement for the exponent can be proven analogously.

For all $n \in \NN$, we can measure the set $S_n := \{G\in \GG \mid \ord(G) = n\}$. We fix an $n$ and define $I_n := \{p \in \PP \mid p \nmid n\}$ and $T_p := \{G\in\GG \mid \pi_p(G) = 0\}$ for all $p \in I_n$. Then $T_p$ is measurable with measure $P(T_p) = P_p(\{0\})= \prod_{i=1}^{\infty}(1-p^{-i}) \leq 1-\frac{1}{p}$. Since $S_n \subseteq T_p$ for all $p\in I_n$, we have for any finite subset $F$ of $I_n$

$$S_n \subseteq \bigcap_{p\in F} T_p.$$

Since $F$ is finite, both sides are measurable and by independence of distinct primes we obtain

$$P(S_n) \leq \prod_{p\in F} P(T_p).$$

The above inequality is true for any finite set $F \subset I_n$, so we may replace the right hand side by the infimum over all such $F$:
\begin{eqnarray*}P(S_n) & \leq & \inf_{F\subset I_n \text{ finite}}\prod_{p\in F} P(T_p) \\
& = & \prod_{p\in I_n} P(T_p) \\
& \leq & \prod_{p\in I_n} (1-\frac{1}{p}) \\
& \leq & \exp\underbrace{\left(\sum_{p\in I_n} (-\frac{1}{p})\right)}_{ = -\infty} \\
& = & 0.\end{eqnarray*}

Therefore, $P(S_n) = 0$ for all $n \in \NN$. But $\GG = \bigcup_{n\in\NN} S_n$, which would imply $P(\GG) = 0$, a contradiction.

\end{proof}

You may wonder why the theorem above only refers to the order and the exponent, but not to the rank. Surprisingly, it turns out that it is even possible to endow $\GG$ with a \emph{probability measure} compatible with the rank. The reason why the rank behaves differently is that it is a uniform quantity in the following sense: If you require the rank of a group $G \in \GG$ to be $r$, then the information that you can extract about the local ranks $r_p$ of $G_p$ is independent of $p$. This seems to be a rather complicated way of saying that essentially the only thing we know for a fixed $p$ is $r_p \leq r$. However, going through the proof of the theorem, this was the crucial point that forbade countable additivity for the order (and the exponent): If we know the order of a group, then we can compute the order of $G_p$ for any $p\in\PP$, so we get \emph{individual} information about local quantities.

This leads us to the definition of the \emph{uniform order} and the \emph{uniform exponent}, which turn out to be better suited for the situation. Afterwards, we will define the notion of \emph{uniform properties} in general.

\subsection{Uniform properties}\label{sect:localquant}

Since we have noticed that the rank behaves better than order and exponent, we want to catch the local behaviour of the rank and transfer it to order and exponent as follows:

\begin{definition}\label{def:localorderrank}
For a prime $p$, we define the \emph{local order on $\GG_p$} as
$$\ord_p(G) := \log_p(\ord(G)).$$

Analogously, we define the \emph{local exponent on $\GG_p$} as
$$\exp_p(G) = \log_p(\min\{n\in\NN^+ \mid n \text{ annihilates G} \}).$$

Now we define the \emph{uniform order $\ord_{uni}$ on $\GG$} and the \emph{uniform exponent $\exp_{uni}$ on $\GG$} as
\begin{eqnarray*}
\ord_{uni}(G) & := & \max_{p\in\PP} \ord_p(G_p)\\
\exp_{uni}(G) & := & \max_{p\in\PP} \exp_p(G_p)
\end{eqnarray*}

\end{definition} 

Note that this definition is completely analogous to the formula 
$$\rk(G)  =  \max_{p\in\PP} \rk(G_p)$$
for the rank. Therefore the ``uniform rank'' coincides with the ordinary rank.

As we will show later, it turns out that there is a probability measure on $\GG$ which allows to measure the uniform order, rank and exponent. So at least we can obtain the minimal program formulated in section \ref{sect:globalquant}, if we work with uniform order, rank, and exponent. But in fact, we can show much more. For this we need a general notion of uniform quantities. For the moment, we restrict ourselves to properties, i.e., to functions $\GG \to \{0,1\}$, telling whether a group has a certain property or not.

\begin{definition}\label{def:localprop}
A \emph{property} (on $\GG$) is a function $E:\GG \to \{0,1\}$. For properties $E_1,E_2$ we define $E_1 \vee E_2$ and $E_1 \wedge E_2$ by 

\begin{eqnarray*}(E_1 \vee E_2)(G) & = & \begin{cases} 1 & \text{ if } E_1(G)=1 \text{ or } E_2(G) = 1, \\ 0 &\text{ otherwise,}\end{cases}, \\
(E_1 \wedge E_2)(G) & = & \begin{cases} 1 & \text{ if } E_1(G)=1 \text{ and } E_2(G) = 1,\\ 0 &\text{ otherwise,}\end{cases}
\end{eqnarray*}

respectively.

A property $E$ is called \emph{uniform} if there is a function, which by abuse of notation we also call $E$, from $\GG_P$ to $\{0,1\}$ such that for all $G\in\GG$

$$E(G) = 1 \text{ if and only if } E(G_p) = 1 \text{ for all } p \in \PP,$$

where $G_p\in\GG_{P}$ corresponds to the $p$-part $G_p$ of $G$ via the identification $\GG_p \stackrel{\cong}{\longleftrightarrow} \GG_{P}$.

If we want to distinguish explicitly between the two functions $E$, then we write $E_{\GG}$ and $E_{\GG_{P}}$, respectively.

Finally, for a uniform property $E$ we define $O(E) := E_{\GG}^{-1}(\{1\})$.
\end{definition}

%
%
%

\begin{remark}~

\begin{itemize}
\item For all uniform properties $E_1$ and $E_2$, we have 
$$O(E_1 \wedge E_2) = O(E_1) \cap O(E_2).$$
\item In general, it is \emph{not} true that $O(E_1 \vee E_2) = O(E_1) \cup O(E_2)$ for local properties $E_1$, $E_2$.
\end{itemize}
\end{remark}

\begin{definition}\label{def:globalmeasure}
Let $\Sigma_{\GG}$ be the coarsest $\sigma$-algebra on $\GG$ that contains the fibers $O(E)$ of all uniform properties $E$ of $\GG$. We define the \emph{Cohen-Lenstra probability measure} $P_{\GG}$ on $\Sigma_{\GG}$ via:

\begin{equation}\label{eq:product1}P_{\GG}(E) := P_{\GG}(E=1) := P_{\GG}(O(E)) := \prod_{p\in\PP} P_p(E_{\GG_{P}}^{-1}(\{1\})),\end{equation} 


where $P_p$ is the Cohen-Lenstra probability on $\GG_{P} \stackrel{\cong}{\longleftrightarrow} \GG_p$. Be aware that for each $p$ we have a different probability measure on $\GG_{P}=\GG_p$. If no confusion with the local Cohen-Lenstra probability measures is possible then we omit the index and write $P$ instead of $P_{\GG}$.
\end{definition}

\begin{remark}~

For any $r\geq 0$, the set of all groups of rank $r$ is measurable with respect to the above measure. This follows from the fact that the property of having rank $\leq r$ is a uniform property, and that
$$\{G\in\GG \mid \rk(G)=r\} = \{G\in\GG \mid \rk(G)\leq r\} \setminus \{G\in\GG \mid \rk(G)\leq r-1\}.$$
 Analogously, the uniform order and uniform exponent are measurable.

\end{remark}

The main result in this chapter is that $P_{\GG}$ is indeed a probability measure on $\GG$ that makes all uniform properties measurable. This justifies many calculations that researchers have carried out without specifying the probability space in which their calculations are supposed to happen. (Of course, the computations were usually carried out in terms of formal series, and the results are definitely true as identities of formal series. But in order to translate the results into probability statements, one needs to specify a probability space.) There are very few statements in the literature which are not uniform statements. There are only two wide-spread non-uniform examples I know of, both of them due to Cohen and Lenstra: 

Firstly, they state that the ``probability'' of a one-element set $\{G_0\}$ is $0$ for every $G_0\in\GG$ \cite[\S 9,II]{CL83}. However, it is obvious that this statement is not compatible with a probability measure, since that would mean that we have a countable probability space with probability $0$ for each atomic event, which is impossible. (Cohen and Lenstra were well aware of the fact that this gives only a content instead of a measure.) Secondly, they state that the ``probability'' that a finite abelian group has $p$-part $G_0$, for a fixed $p$-group $G_0$, is $P_p(G_0)$. This is highly problematic. As we have seen before, there is no probability measure on $\GG$ which is compatible with this statement, so we should at least avoid talking about probabilities in this context.

Now let us come to the main theorem:

\begin{theorem}\label{thm:globalmeasure}
The Cohen-Lenstra probability measure $P_{\GG}$ is indeed a probability measure, and it makes all uniform properties measurable.
\end{theorem}

The proof is complicated and the whole next section is devoted to it.

Before we come to the proof, let me first summarize our discussion about measurable functions: The theorem asserts that the rank, the uniform order, the uniform exponent and all other uniform properties are measurable, and so are all functions defined in these terms, for example, the expected value or higher moments of these functions.

\emph{Not} measurable are the classical order and exponent, and the property that the $p$-part of a group is isomorphic to some fixed $p$-group $G_0$. But for any single one of these properties we have shown (cf.\!\! theorem \ref{thm:globalorderandexponent} and page \pageref{plain:measurelocalstuff}, respectively) that there is no probability measure which would make these functions measurable, so we could not expect to be able to measure these functions. More generally, essentially no function is measurable that is defined via the $p$-part of the group, for some fixed $p$.

\section{The existence of a global measure}\label{sect:globalmeasure}

In this section, we prove theorem \ref{thm:globalmeasure}. We proceed as follows: First, we construct an outer measure on the power set of $\GG$ that coincides on certain key sets with our desired probability measure $P_{\GG}$. Then we use the theorem of Carath\'{e}odory to deduce the existence of a $\sigma$-algebra of measurable sets such that the outer measure is a measure on these sets. Finally we show that uniform properties are measurable with respect to this $\sigma$-algebra.

Let me start with some general remarks. First of all, note that the product that defines $P_{\GG}$ consists only of factors $\leq 1$. Therefore, we either have absolute convergence or we have definite divergence to $0$. In both cases, we may arbitrarily reorder the factors, and we may apply the formula

$$\prod_{i}a_i = \exp\left( \sum_i \log(a_i)\right).$$

Since this is a major tool for us, we will be concerned about estimating $\log(a_i)$. We will use the formula

$$-2h \leq \log(1-h) \leq -h,$$

which is true for any $0 \leq h \leq \frac{1}{2}$ (by Jensen's inequality) and in particular for $h=\frac{1}{p}$, for any prime $p$.
%

\subsection{First properties of the global measure}

This section contains essentially some technical lemmas about $P = P_{\GG}$. However, lemma \ref{lem:positiveprob} is of intrinsic interest, independent of its use in the construction of the probability space.\medskip

So let us check a couple of properties of $P$. First of all, in definition \ref{def:globalmeasure} we have not excluded the case that $E(0) = 0$, where on the left hand side $0$ stands for the trivial partition. But in this case $O(E)$ is empty, since any group has trivial $p$-parts for almost all $p\in\PP$. In other words, we have non-trivial ways to describe the empty set, so the formula in \ref{def:globalmeasure} had then better give $P(E) = P(\emptyset) = 0$, if it is supposed to make sense. Indeed this is the case:

\begin{lemma}
If $E$ is a uniform property with $E(0) = 0$, then $P(E) = 0$.
\end{lemma}

\begin{proof}

We have $E_{\GG_{P}}^{-1}(\{1\}) \subseteq \GG_{P} \setminus \{0\}$, so we have 
\begin{eqnarray*}
P_p(E_{\GG_{P}}^{-1}(\{1\})) & \leq & P_p(\GG_p \setminus \{0\})\\
& = & 1-P_p(\{0\})\\
& = & 1- \prod_{i=1}^{\infty}(1-p^{-i})\\
& \leq & 1-\left(1-2\sum_{i=1}^{\infty}p^{-i}\right)\\
& = & 2\sum_{i=1}^{\infty}p^{-i}\\
& = & \frac{2}{p-1}.
\end{eqnarray*}

Therefore,

\begin{eqnarray*}
P(E=1) & = & \prod_{p\in\PP}P_p(E_{\GG_{P}}^{-1}(\{1\})) \\
& \leq & \prod_{p\in\PP}\frac{1}{p-1} \\
& = & 0.
\end{eqnarray*}

\end{proof}

So from now on we may assume that $E(0) = 1$.

We continue with a lemma, which is of interest in its own right:

\begin{lemma}
Let $E$ be a uniform property with $E(1) = 0$, where $1$ is the unique partition of $1$. Then $P(E) = 0$.
\end{lemma}

\begin{proof}

We have $E_{\GG_{P}}^{-1}(\{1\}) \subseteq \GG_{P} \setminus \{1\}$, so we get

\begin{eqnarray*}
P_p(E_{\GG_{P}}^{-1}(\{1\})) & \leq & P_p(\GG_{P} \setminus \{1\})\\
& = & 1-P_p(\{1\})\\
& = & 1- \frac{1}{p-1}\prod_{i=1}^{\infty}(1-p^{-i})\\
& = & 1- p^{-1}\prod_{i=2}^{\infty}(1-p^{-i})\\
& \leq & 1-p^{-1}\left(1-2\sum_{i=2}^{\infty}p^{-i}\right)\\
& = & 1-p^{-1}+\frac{2p^{-2}}{p-1}\\
& \stackrel{\text{for } p>2}{\leq} & 1-\frac{1}{2}p^{-1}.
\end{eqnarray*}

Therefore,

\begin{eqnarray*}
P(E=1) & = & \prod_{p\in\PP}P_p(E_{\GG_{P}}^{-1}(\{1\})) \\
& \leq & \prod_{p\in\PP\setminus\{2\}}\left(1-\frac{1}{2}p^{-1}\right) \\
& = & \exp\left(\sum_{p\in\PP\setminus\{2\}}\log\left(1-\frac{1}{2}p^{-1}\right)\right)\\
& \leq & \exp\underbrace{\left(\sum_{p\in\PP\setminus\{2\}}\left(-\frac{1}{2}p^{-1}\right)\right)}_{= -\infty}\\
& = & 0.
\end{eqnarray*}

\end{proof}

In fact, we even have equivalence:

\begin{lemma}\label{lem:positiveprob}
Let $E$ be a uniform property. Then $P(E) > 0$ if and only if $E(0) = E(1) = 1$.
\end{lemma}

\begin{proof}

We have already shown one direction, so now assume that the latter statement is true. Then we have $E^{-1}(\{1\}) \supseteq \{0, 1\}$, so we get

\begin{eqnarray*}
P_p(E^{-1}(\{1\})) & \geq & P_p(\{0,1\})\\
& = & \left(\sum_{i=0}^{\infty}p^{-i}\right)\prod_{i=1}^{\infty}(1-p^{-i})\\
& = & \frac{1}{1-p^{-1}}\prod_{i=1}^{\infty}(1-p^{-i})\\
& = & \prod_{i=2}^{\infty}(1-p^{-i}).
\end{eqnarray*}

Therefore,

\begin{eqnarray*}
P(E) & = & \prod_{p\in\PP}P_p(E^{-1}(\{1\})) \\
& \geq & \prod_{p\in\PP} \prod_{i=2}^{\infty}(1-p^{-i})\\
& = &  \prod_{i=2}^{\infty}\prod_{p\in\PP}(1-p^{-i})\\
& = & \prod_{i=2}^{\infty}\zeta^{-1}(i),
\end{eqnarray*}

where $\zeta$ denotes the Riemann $\zeta$-function.

The latter product is well-known and converges against a positive constant $\approx 0.435757...$ (see e.g. \cite[\S 7]{CL83}).

\end{proof}

\subsection{The global outer measure}

In order to define an outer measure, we first need to specify a family $\DDD$ of subsets with non-negative values (``Method I'' in \cite{Mun53}).

\begin{definition}\label{def:OandD}
Let $E_1,\ldots,E_r$ be uniform properties. In accordance with the former definition of $O(E)$ we define 

$$O(E_1,\ldots,E_r) := \bigcup_{i=1}^r E_i^{-1}(1),$$

and we set 
$$\OOO := \{O(E_1,\ldots,E_r) \mid r \geq 0, E_1,\ldots, E_r \text{ uniform properties}\}.$$

Let $E_1,\ldots,E_r,F_1,\ldots,F_s$ be uniform properties. Then we define 

$$D(E_1,\ldots,E_r;F_1,\ldots,F_s) := O(E_1,\ldots,E_r) \setminus O(F_1,\ldots,F_s),$$

and we set 
\begin{eqnarray*}\DDD & := & \{ D(E_1,\ldots,E_r;F_1,\ldots,F_s) \mid r,s \geq 0,\\&& E_1,\ldots, E_r,F_1,\ldots,F_s  \text{ uniform properties}\}.\end{eqnarray*}

By slight abuse of notation, I will sometimes write $O(\EE)$ and $D(\EE;\FFF)$ instead of $O(E_1,\ldots,E_r)$ and $D(E_1,\ldots,E_r;F_1,\ldots,F_s)$, respectively, where $\EE$ and $\FFF$ are the families $\{E_1,\ldots,E_r\}$ and $\{F_1,\ldots,F_s\}$.

By even stronger abuse of notation, I will occasionally write $O(E_i)$ and $D(E_i;F_j)$ in these cases.
\end{definition}

\begin{remark}~\vspace{-.5ex}

\begin{itemize}
\item $\OOO$ is embedded into $\DDD$ by setting $s:=0$.
\item For all uniform properties $E_1,\ldots, E_r$ and $E_1',\ldots,E_s'$:
 \begin{eqnarray*}O(E_1,\ldots, E_r) \cap O(E_1',\ldots,E_s') & = & O(E_1 \wedge E_1', E_1 \wedge E_2', \ldots, E_r \wedge E_s').\\
 O(E_1,\ldots,E_r) \cup O(E_1',\ldots,E_s') & = & O(E_1,\ldots,E_r,E_1',\ldots,E_s').\end{eqnarray*}
\item For all uniform properties $E_1,\ldots, E_r,F_1,\ldots,F_s$ and $E_1',\ldots,E_t'$:
 \begin{align*}
 \begin{split}D(E_1,\ldots, E_r;F_1,\ldots,F_s) \cap O(E_1',\ldots,E_t') & \\ = D(E_1 \wedge E_1', \ldots &, E_r \wedge E_t';F_1,\ldots,F_s).\end{split}\end{align*}
\emph{Caution:} No similar formula for the union exists.
\item $\DDD$ is closed under intersection. More precisely, we have
\begin{eqnarray*} && D(E_1,\ldots,E_{r_1};F_1,\ldots,F_{s_1}) \cap D(\tilde{E}_1,\ldots,\tilde{E}_{r_2};\tilde{F}_1,\ldots,\tilde{F}_{s_2}) \\
&&= D(E_1 \wedge \tilde{E}_1, E_1 \wedge \tilde{E}_2,\ldots,E_{r_1}\wedge \tilde{E}_{r_2};F_1,\ldots,F_{s_1},\tilde{F}_1,\ldots,\tilde{F}_{s_2}) \end{eqnarray*}
\item $\DDD$ is \emph{not} closed under union!
\item $\DDD$ is \emph{not} closed under set difference!\pagebreak[2]
\end{itemize}
 
\end{remark}

We would like to exend the definition of the function $P$ from single uniform properties to the whole set $\DDD$. In order to do so, we need one more remark:

\begin{remark}\label{rem:normalformofD}~

\begin{itemize}
\item We may always assume that the defining uniform properties $E_1,\ldots,E_r,$ $F_1,\ldots,F_s$ of a set $D(E_i,F_j) \in \DDD$ satisfy the condition $$O(F_1,\ldots,F_s) \subseteq O(E_1,\ldots,E_r).$$

In fact, if $E_i$ and $F_j$ are uniform properties which do \emph{not} satisfy this condition, then we may use the identity 
$$D(E_1,\ldots,E_r;F_1,\ldots,F_s) = D(E_1,\ldots,E_r;F_1\wedge E_1, F_1\wedge E_2,\ldots,F_s\wedge E_r)$$

to enforce the condition.
\item It is easy to see that $O(F_1,\ldots,F_s) \subseteq O(E_1,\ldots,E_r)$ if and only if for any $i \in \{1,\ldots,s\}$ there is a $j \in \{1,\ldots,r\}$ such that $O(F_i) \subseteq O(E_j)$. 
\end{itemize}
\end{remark}

\begin{defiprop}\label{defiprop:Pwelldefined}
We extend $P$ to $\DDD$ as follows: We have already defined $P(O(E))$ for a single uniform property $E$ in definition \ref{def:globalmeasure}. Because of the formula $O(E_1) \cap O(E_2) = O(E_1 \wedge E_2)$ the function $P$ is also defined on intersections of sets in $\OOO$. Hence we may extend $P$ to sets of the form $O(E_1,E_2)$ ($=O(E_1) \cup O(E_2)$) via  
$$P(O(E_1,E_2)) := P(O(E_1)) + P(O(E_1)) - P(O(E_1\wedge E_2)).$$

Continuing inductively, we extend $P$ on the set $\OOO$. Finally, for uniform properties $E_1,\ldots,E_r,F_1,\ldots,F_s$ with $O(F_1,\ldots,F_s) \subseteq O(E_1,\ldots,E_r)$ we set 
$$P(D(E_1,\ldots,E_r;F_1,\ldots,F_s)) := P(O(E_1,\ldots,E_r)) - P(O(F_1,\ldots,F_s)).$$

This yields a well-defined map $P:\DDD\rightarrow [0,1]$.
\end{defiprop}

\begin{proof}
The procedure for computing $P(O(E_1,\ldots,E_r))$ yields the Inclusion-Exclusion Formula, which is independent of the order of the $E_i$. So we only need to show that whenever 
\begin{equation}\label{eq:Pwelldefined}D(E_1,\ldots,E_{r_1};F_1,\ldots,F_{s_1}) = D(E_1',\ldots,E_{r_2}';F_1',\ldots,F_{s_2}')\end{equation}
 then the value of $P$ coincides for both sets. 

Let us first consider the case that $O(E_1,\ldots,E_{r_1}) = O(E_1',\ldots,E_{r_2}')$. If for some $i,j$ we have $O(E_i) \subseteq O(E_j)$, then the result of the Inclusion-Exclusion Formula does not change if we omit $E_i$. So we may assume that all $E_i$ are maximal in the sense that $O(E_i)$ is not a proper subset of $O(E_j)$, for all $j \neq i$. We assume the same for the $E_i'$. Then I claim that $E_1$ occurs also on the right hand side. By symmetry, this will imply the statement for $\OOO$.

Because of the maximality of $E_1$, it suffices to show that $O(E_1) \subseteq O(E_i')$ for some $i$. (Then by symmetry, $O(E_i') \subseteq O(E_j)$ for some $j$, and by maximality of $E_1$ we conclude $j=1$ and $O(E_1)=O(E_i')$). Assume not. Then for all $1\leq i \leq r_2$ there is a partition $\underline{n}_i$ such that $E_1(\underline{n}_i) = 1$ and $E_i'(\underline{n}_i) = 0$. Now take $r_2$ distinct primes $p_1,\ldots,p_{r_2}$ and consider a group with $p_i$-part equal to $\underline{n}_i$, for $i=1,\ldots,r_2$. Then this group is contained in $O(E_1)$ but in none of the $O(E_i')$, contradicting $O(E_1,\ldots,E_{r_1}) = O(E_1',\ldots,E_{r_2}')$.

This finishes our proof for $\OOO$. For $\DDD$, first notice that by the preceding remark, $P$ is indeed defined on the whole set $\DDD$. To show that it is well-defined we use essentially the same argument as for $\OOO$. But beforehand, we replace each property $F_i$ by properties $F_{i,1} := F_i \cap E_1,\ldots, F_{i,r_1} := F_i \cap E_{r_1}$. Since this does not change $O(F_{\dots})$, it does not affect $P$. Now we may further assume that no $E_i$ equals an $F_j$. Otherwise, we replace the tuple 
$$(E_1,\ldots,E_{r_1};F_{1,1},\ldots,F_{s_1,r_1})$$
by 
$$(E_1,\ldots,\widehat{E_i},\ldots,E_{r_1};F_{1,1},\ldots,\widehat{F_{i,1}},\widehat{\ldots},\widehat{F_{i,r_1}},\ldots,F_{s_1,r_1}),$$
where a hat indicates that the entry is removed. (The change of the $F$ is necessary to ensure that each $O(F)$ is still contained in some $O(E)$). You can easily check that this procedure does not change the value of $P$.

Furthermore, we may assume that all $E_i$, $E_i'$ are maximal and all $F_i$, $F_i'$ are maximal (in the sets $\{F_j\}, \{F_j'\}$, respectively). If not, then remove the superfluous sets.

Now we proceed as in the proof for $\OOO$. First we show that the $E_i$ and the $E_i'$ coincide. Assume $E_1$ does not appear in the right hand side. Choose mutually distinct primes $p_i, p_{i,j}$ for each $E_i'$ and each $F_{i,j}$, respectively. Then construct a group such that its $p_i$-part corresponds to a partition in $E_1^{-1}(1) \setminus E_i'^{-1}(1)$ and its $p_{i,j}$-part corresponds to a partition in $E_1^{-1}(1) \setminus F_{i,j}^{-1}(1)$. The assumptions above ensure that the latter sets are all non-empty. Then the group is in $E_1$, but it is neither in any $E_i'$ nor in any $F_{i,j}$. Therefore, it is contained in the left hand side, but not in the right hand side of \eqref{eq:Pwelldefined}. Contradiction! So the assumption was wrong, and the $E_i$ and the $E_j'$ coincide.

Now turn to the $F_{i,j}$ and $F_{i,j}'$. Since $O(E_i) = O(E_i')$, $O(F_{i,j})\subseteq O(E_i)$, $O(F_{i,j}')\subseteq O(E_i')$, and $O(E_i)\setminus O(F_{i,j}) = O(E_i')\setminus O(F_{i,j}')$, we can deduce $O(F_{i,j}) = O(F_{i,j}')$. Now we may apply the first part of the proof (for $\OOO$) to conclude that $P(O(F_{i,j})) = P(O(F_{i,j}'))$. Putting things together, we see that $$P(D(E_1,\ldots,E_{r_1};F_1,\ldots,F_{s_1})) = P(D(E_1',\ldots,E_{r_2}';F_1',\ldots,F_{s_2}')),$$
as required.
\end{proof}

\begin{remark}~

In the following proofs (as well as in the proof above), be aware that the formula 
$$P(D(E_1,\ldots,E_r;F_1,\ldots,F_s)) = P(O(E_1,\ldots,E_r))- P(O(F_1,\ldots,F_s))$$
is \emph{not true} if we omit the condition $$O(F_1,\ldots,F_s) \subseteq O(E_1,\ldots,E_r).$$

\end{remark}\medskip

We will use the function $P$ to define an outer measure. But before that, we prove a technical lemma about $P$:

\begin{lemma}\label{lem:disjointunion}~

\begin{enumerate}
\item Let $D_1,\ldots,D_n \in \DDD$ be mutually disjoint, and let $D_0\in \DDD$ be such that 
$$\bigcup_{i=1}^n D_i \subseteq D_0.$$

Then 

$$\sum_{i=1}^{n} P(D_i) \leq P(D_0).$$

In particular, this implies that $P$ is monotone, i.e., for $D_1 \subseteq D_0$ we have $P(D_1) \leq P(D_0)$.
\item Let $D_0, D_1,\ldots,D_n \in \DDD$ be such that
$$D_0 \subseteq \bigcup_{i=1}^n D_i.$$

Then 

$$P(D_0) \leq \sum_{i=1}^{n} P(D_i).$$

\end{enumerate}
\end{lemma}

\begin{proof}
We only prove the first statement, which is slightly more complicated. The proof of the second case is completely analogous, except that we do not have to worry about the $D_i$ being disjoint.\medskip

Let $D_i = D(\EE_i,\FFF_i)$ for all $i=0,\ldots,n$, where $\EE_i,\FFF_i$ are collections of uniform properties. We may assume $O(\FFF_i)\subseteq O(\EE_i)$ for all $i$. Then $P(D_i) = P(O(\EE_i)) - P(O(\FFF_i))$ for all $i=0,\ldots,n$.

Therefore, we need to show that 

\begin{equation}
\sum_{i=1}^{n} P(O(\EE_i)) - \sum_{i=1}^n P(O(\FFF_i)) \leq  P(O(\EE_0)) - P(O(\FFF_0)),
\end{equation}
or equivalently by expanding the $P(O(\EE_i))$:

\begin{align}\label{eq:OOO-equation}
 \begin{split}
\sum_{i=1}^{n} \sum_{S\subseteq \EE_i} (-1)^{\#S} P\big(\bigwedge_{E\in S} E\big) - & \sum_{i=1}^n \sum_{S\subseteq \FFF_i} (-1)^{\#S} P\big(\bigwedge_{F\in S} F\big) \\
\leq \sum_{S\subseteq \EE_0} (-1)^{\#S} & P\big(\bigwedge_{E\in S} E\big) - \sum_{S\subseteq \FFF_0} (-1)^{\#S} P\big(\bigwedge_{F\in S} F\big).
\end{split} &&
\end{align}

Let us first examine the prerequisites of the statement. We may assume that no $E \in \EE_i$ is contained in any $F \in \FFF_i$, for $i=0,\ldots,n$. Then it is easy to see that the prerequisites are satisfied if and only if the following conditions are satisfied:

\begin{enumerate}
\item[1.] $O(\EE_i) \subseteq O(\EE_0)$ for $i= 1,\ldots,n$.
\item[2.] $O(\EE_i) \cap O(\EE_j) \subseteq O(\FFF_i) \cup O(\FFF_j)$ for all $1\leq i < j \leq n$.
\item[3.] $O(\FFF_0) \subseteq O(\FFF_i)$ for $i=1,\ldots,n$.
\end{enumerate}

Now let $\PP_{\leq x} := \{p\in\PP \mid p \leq x\}$ and let 
$$\GG_{\leq x} := \prod_{p\in \PP_{\leq x}}\GG_p.$$ 

Then $\GG_{\leq x}$ is the direct product of probability spaces and carries a unique product probability measure. The set $\GG_{\leq x}$ embeds naturally into $\GG$. So for each uniform property $E$, we may define $O_{\leq x}(E) := O(E) \cap \GG_{\leq x}$. By definition of the product probability, we have for these sets the probabilities
 
$$P_{\leq x}(E) := P_{\GG_{\leq x}}(O_{\leq x}(E)) = \prod_{p\in\PP_{\leq x}} P_p(E).$$

Then it is evident that for any uniform property $E$,

$$P(E) = \lim_{x\to\infty} P_{\leq x}(E).$$

Conditions 1.--3. are still satisfied if we intersect both sides with $\GG_{\leq x}$, so we also have

\begin{enumerate}
\item[1'.] $O_{\leq x}(\EE_i) \subseteq O_{\leq x}(\EE_0)$ for $i= 1,\ldots,n$.
\item[2'.] $O_{\leq x}(\EE_i) \cap O_{\leq x}(\EE_j) \subseteq O_{\leq x}(\FFF_i) \cup O_{\leq x}(\FFF_j)$ for all $1\leq i < j \leq n$.
\item[3'.] $O_{\leq x}(\FFF_0) \subseteq O_{\leq x}(\FFF_i)$ for $i=1,\ldots,n$.
\end{enumerate}
 
Now for sufficiently large $x$ (we need more primes than uniform properties involved), the conditions 1.'--3.' are equivalent to the statement 

\begin{eqnarray*}& & D_1\cap\GG_{\leq x},\ldots,D_n\cap\GG_{\leq x} \text{ are mutually disjoint, and }\\
& & \bigcup_{i=1}^n D_i \cap\GG_{\leq x} \subseteq D_0 \cap\GG_{\leq x}. \end{eqnarray*}

Since $\GG_{\leq x}$ is a probability space, we deduce 

$$\sum_{i=1}^{n} P(D_i \cap\GG_{\leq x}) \leq P(D_0 \cap\GG_{\leq x}),$$

or equivalently 

\begin{align*}
\sum_{i=1}^{n} \sum_{S\subseteq \EE_i} (-1)^{\#S} P_{\leq x}\big(\bigwedge_{E\in S} E\big) - \sum_{i=1}^n \sum_{S\subseteq \FFF_i} (-1)^{\#S} P_{\leq x}\big(\bigwedge_{F\in S} F\big) & \\ \leq  \sum_{S\subseteq \EE_0} (-1)^{\#S} P_{\leq x}\big(\bigwedge_{E\in S} E\big) - \sum_{S\subseteq \FFF_0} (-1)^{\#S} & P_{\leq x}\big(\bigwedge_{F\in S} F\big).
\end{align*}

Since we have finite sums and differences on both sides, we obtain equation \eqref{eq:OOO-equation} by taking the limit $x\to\infty$.
This proves the claim.

\end{proof}

Now we come to the definition of the outer measure:

\begin{definition}
For any $A \subset \GG$, we define the \emph{outer measure $\nu$} as 
$$\nu(A) := \inf\left\{\sum_{i=1}^{\infty}P(A_i)\left| A_i \in \DDD \text{ and } A \subset \bigcup_{i=1}^{\infty} A_i\right.\right\}.$$
\end{definition}

\begin{remark}
The definition above always yields an outer measure, for any map $P:S\rightarrow [0,\infty]$, where $S$ is any subset of the power set of $\GG$ containing $\emptyset$ and $P(\emptyset) = 0$ \cite{Mun53}.

Recall that an outer measure is almost a measure, only we replace the $\Sigma$-additivity by $\Sigma$-subadditivity. More precisely, an outer measure on a space $X$ is a function $\nu$ from the power set of $X$ into the interval $[0,\infty]$ satisfying the three conditions:
\begin{itemize}
\item $\nu(\emptyset) = 0$.
\item \emph{Monotonicity}: $\nu(A) \leq \nu(B)$ for all $A \subseteq B \subseteq X$.
\item \emph{$\Sigma$-subadditivity}: $$\nu\big(\bigcup_{i=1}^{\infty} A_i\big) \leq \sum_{i=1}^{\infty}\nu(A_i)$$
for all $A_i\subseteq X$.
\end{itemize}

\end{remark}

\subsection{The global measure}

Next we check that $\nu$ and $P$ coincide on $\DDD$. We divide up the proof into several steps. First, we prove a helpful lemma:

\begin{lemma}\label{lem:onDDD}
Let $D = D(E_1,\ldots,E_r;F_1,\ldots,F_s) \in \DDD$ such that for all $1\leq i\leq r$ we have finite fibers $E_i^{-1}(1)$, and assume without loss of generality that $O(E_i) \not\subseteq O(F_{i'})$ for all $i,i'$. Let $\tilde{D}_j = D(\tilde{E}_{j,k};\tilde{F}_{j,k'})$ be an arbitrary family in $\DDD$ such that 

$$D \subseteq \bigcup_{j}\tilde{D}_j.$$

Then for each $i$ there exists a $j$ such that $O(E_i) \subseteq O(\tilde{E}_{j,1}, \tilde{E}_{j,2},\ldots)$ and such that $O(E_i) \not\subseteq O(\tilde{F}_{j,1},\tilde{F}_{j,2},\ldots)$.
\end{lemma}

\begin{proof}
We use a similar argument as in the proof of \ref{defiprop:Pwelldefined}. Assume that the assertion is wrong for some $i$. Then for all $j$ there exists an $\underline{n} \in E_i^{-1}(1)$ such that $\underline{n} \notin \tilde{D}_j$, and in particular $\underline{n} \notin O(\tilde{E}_{j,1}, \tilde{E}_{j,2},\ldots)$.

Now choose mutually distinct primes $p_{\underline{n}}$ for each $\underline{n} \in E_i^{-1}(1)$. Consider a group $G$ with $p_{\underline{n}}$-part $\underline{n}$ for all $\underline{n} \in E_i^{-1}(1)$. Then $G \in O(E_i)$, but $G \notin O(F_{i'})$ for all $i'$, since otherwise $O(E_i) \subseteq O(F_{i'})$. 

Furthermore, $G \notin D(\tilde{E}_{j,1}, \tilde{E}_{j,2},\ldots;\tilde{F}_{j,1}, \tilde{F}_{j,2},\ldots)$ for all $j$, contradicting the prerequisite $D \subset \bigcup_{j}\tilde{D}_j$. This proves the lemma.
\end{proof}

Next we prove that $P$ and $\nu$ coincide on a certain subset of $\DDD.$

\begin{lemma}\label{lem:finite}
Let $D$ be the finite disjoint union of sets 
$$D^{(k)} = D(E_1^{(k)},\ldots,E_{r_k}^{(k)};F_1^{(k)},\ldots,F_{s_k}^{(k)}) \in \DDD$$
such that the fibers $(E_i^{(k)})^{-1}(1)$ are finite for all $i$ and $k$. Then we have 
$$\nu(D) = \sum_kP(D^{(k)}).$$
\end{lemma}

\begin{proof}
We set $D(E_1,\ldots,E_r;F_1,\ldots,F_s) < D(E_1',\ldots,E_r';F_1',\ldots,F_s')$ if and only if $O(E_1,\ldots,E_r) \subsetneq O(E_1',\ldots,E_r')$; in this way, we impose a partial ordering on $\DDD$. It is a well-ordering on sets in $\DDD$ with finite fibers $E_i^{-1}(1)$, and the lexicographic ordering extends this to the set of all finite tuples $(D^{(k)})$ of elements of $\DDD$ with finite fibers $E_i^{-1}(1)$. The lexicographic ordering is still a well-ordering so we may use induction with respect to this ordering. 

For the sake of clarity, I will restrict the proof to the case where we have only one set $D^{(k)}$ and simplify the notation to $D= D(E_1,\ldots,E_r;F_1,\ldots,F_s)$. The extension to a disjoint union is straightforward: Just apply the descending step to the largest (possibly several) $D^{(k)}$'s with respect to the ordering.

As usual we assume that the sets $O(E_i)$ are mutually not contained in each other.

If $D$ is minimal then $D=\emptyset$ and we have $P(D) = 0 = \nu(D)$.

So assume $D\neq \emptyset$. Since $D\in \DDD$ we have $\nu(D) \leq P(D)$. So we only need to show that for any countable familiy $A_j$ in $\DDD$ with $D \subset \bigcup_{j=1}^{\infty} A_j$ we have $P(D) \leq \sum_{j=1}^{\infty}P(A_j)$.

Let $A_j$ be such a family. Consider $E_1$. If $E_1$ is contained in any of the sets $F_1,\ldots,F_s$, then we simply omit it and we are done by induction hypothe\-sis. So assume otherwise. Then by lemma \ref{lem:onDDD} there exists an index $j_0$, $A_{j_0} = D(\tilde{E}_k;\tilde{F}_k) =: D(\tilde{\EE};\tilde{\FFF})$, such that 
\begin{eqnarray*}
O(E_1) & \subseteq & O(\tilde{\EE})\text{, and}\\
O(E_1) & \not\subseteq & O(\tilde{\FFF}).
\end{eqnarray*}

Now consider the set $D_0 := D\setminus A_{j_0}$. We will see that we need to compute the measure of this set. Unfortunately, $D_0$ is not in $\DDD$ in general, but it is the disjoint union of two elements in $\DDD$. Basically we will use the decomposition
$$D_0 = D \setminus A_{j_0} = D\setminus (O(\tilde{\EE}) \setminus O(\tilde{\FFF})) = (\underbrace{D \setminus O(\tilde{\EE})}_{=: D_1}) \dot{\cup} (\underbrace{D \cap O(\tilde{\FFF})}_{=:D_2}).$$

We need to show that $D_1,D_2 \in \DDD$: We write $D_1 = D(\EE_1;\FFF_1)$, where
\begin{eqnarray*}
\EE_1 & := & \{E_2,\ldots,E_r\} \cup \{ E_1\wedge \tilde{F}_{1}, E_1\wedge \tilde{F}_{2}, \ldots\}\text{, and}\\
\FFF_1 & := & \{F_1,\ldots,F_s\} \cup \{\tilde{E_1},\tilde{E_2},\ldots\},
\end{eqnarray*}

and $D_2 = D(\EE_2;\FFF_2)$, where
\begin{eqnarray*}
\EE_2 & := & \{ E_i\wedge \tilde{F}_{i'} \mid i,i' = 1,2,\ldots\}\text{, and}\\
\FFF_2 & := & \{F_1,\ldots,F_s\}.
\end{eqnarray*}

Then $D_1$ and $D_2$ are disjoint with union $D\setminus A_{j_0}$, they have finite fibers and are strictly smaller (in the inductive sense) than $D$, so we may apply the induction hypothesis and conclude $\nu(D_1 \cup D_2) = P(D_1)+P(D_2)$. 

%
%
%

Since $D_1 \cup D_2 = D \setminus A_{j_0}$, we have 
$$D_1 \cup D_2 \subset \bigcup_{\atop{j=1}{j\neq j_0}}^{\infty} A_j,$$
and consequently
$$\nu(D_1 \cup D_2) \leq \sum_{\atop{j=1}{j\neq j_0}}^{\infty}P(A_j).$$

Now we can put everything together: Reusing the formula $D \subseteq D_1 \cup D_2 \cup A_{j_0}$, we see that $P(D) \leq P(D_1) + P(D_2) + P(A_{j_0})$ by lemma \ref{lem:disjointunion}, and therefore

\begin{eqnarray*}
P(D) & \leq & P(D_1) + P(D_2) + P(A_{j_0})\\
& = & \nu(D_1 \cup D_2) + P(A_{j_0})\\
& \leq & \left(\sum_{\atop{j=1}{j\neq j_0}}^{\infty}P(A_j)\right) + P(A_{j_0})\\
& = & \sum_{j=1}^{\infty}P(A_j).
\end{eqnarray*}

This proves $P(D) \leq \nu(D)$, as required.

\end{proof}

Now we are ready to tackle the general case:

\begin{prop}\label{prop:nuP}
For any $D \in \DDD$, we have $\nu(D) = P(D)$.
\end{prop}

\begin{proof}
Let $D = D(E_1,\ldots,E_r;F_1,\ldots,F_s)$. Let $\EE$ be the $r$-tuple $(E_1,\ldots,E_r)$ and let $\FFF$ be the $s$-tuple $(F_1,\ldots,F_s)$. In the following, $\EE'$ will always denote an $r$-tuple of uniform properties that is finite in the sense that the fibers $E'^{-1}(1)$ are finite for all properties $E'$ in $\EE'$. We shall write $\EE' \leq \EE$ if for all $1\leq i \leq r$ we have $O(E'_i) \subseteq O(E_i)$.

The crucial step in this proof is to show

\begin{equation}\label{eq:finiteexhaustion}P(D) = \sup_{\EE' \leq \EE \text{ finite}} P(D(\EE',\FFF)).\end{equation}

The inequality ``$\geq$'' is trivial. For the other direction, note that for any finite $\EE' \leq \EE$, we have

$$P(D(\EE,\FFF)) - P(D(\EE',\FFF)) \leq P(O(\EE)) - P(O(\EE')).$$

Therefore, it suffices to show that $P(O(\EE)) = \sup_{\EE'} P(O(\EE'))$.

Furthermore, it suffices to consider the case $r=1$ (i.e., $\EE$ consists of only one uniform property), because by the Inclusion-Exclusion Formula $P(O(\EE))$ can be computed as a finite sum (with signs) from values $P(O(E))$, where $E$ is a single uniform property.

Altogether, we need to show that for each uniform property $E$, we have

$$P(O(E)) = \sup_{E' \leq E \text{ finite}} P(O(E')),$$

where ``$E' \leq E$ finite'' means that $E'^{-1}(1)\subseteq E^{-1}(1)$ and  $E'^{-1}(1)$ is finite.

We may assume that $P(O(E)) > 0$, otherwise the statement is trivial. 

%
Let us look at the local situation: Let $p\in\PP$ and let $n_0\in\NN$. For any $n\in \NN^+$, it is possible to choose $E' \leq E$ finite such that $w_p(E) \leq w_p(E') + \sum_{i=n}^{\infty}a_iq^i$ ($a_i =$ number of partitions of $i$) \emph{as power series}, i.e., coefficient-wise. 

By lemma \ref{lem:partgrowth} we know that $a_i \in O(\phi^i)$, where $\phi = 1.618\ldots$ is the golden ratio. There exists a constant $d<1$ (e.g., $d:=0.7$) such that $2^d > \phi$ and such that $2^{4-d} > 2^3+1$. Then it is easy to see that for all primes $p$ we have $p^{4-d} > p^3+1$. By choosing $n$ large enough, we may further assume that $a_i \leq 2^{di-n_0-3}$ for all $i\geq n$. Then in particular $a_i \leq p^{di-n_0-3}$ for all primes $p$. Also by lemma \ref{lem:positiveprob}, we may assume that $P(O(E')) \geq c$ for some $c>0$, and therefore also $P_p(E') \geq P(O(E')) \geq c$ for all $p\in\PP$.

Then we have 
\begin{eqnarray*} 
w_p(E) & \leq & w_p(E') + \sum_{i=n}^{\infty}a_iq^{i} \\
&\leq & w_p(E') + \sum_{i=n}^{\infty}p^{di-n_0-3}p^{-i} \\
&= & w_p(E') + p^{-n_0-3}\sum_{i=n}^{\infty}p^{(d-1)i} \\
&= & w_p(E') + p^{-n_0}p^{-3}\frac{p^{n(d-1)}}{1-p^{d-1}}\\
&= & w_p(E') + p^{-n_0}\frac{p^{(n-1)(d-1)}}{p^{4-d}-p^3}\\
& \leq & w_p(E') + p^{-n_0},
\end{eqnarray*} 

where in the last inequality we use that the fraction has numerator $\leq 1$ and denominator $\geq 1$.

For the probability, we must multiply with $\prod_{i=0}^{\infty} (1-p^{-i})$:

$$P_p(E) \leq P_p(E') + p^{-n_0}\prod_{i=0}^{\infty} (1-p^{-i}) \leq P_p(E') + p^{-n_0}$$

Since our choice of $E'$ and of $n$ was independent of $p$, the analysis works for all $p$. Putting this together, we get

\begin{eqnarray*} 
P(E) & = & \prod_{p\in\PP} P_p(E)\\
& \leq & \prod_{p\in\PP} \left(P_p(E') + p^{-n_0}\right)\\
& = & \left(\prod_{p\in\PP} P_p(E')\right) \prod_{p\in\PP}\left(1+\frac{p^{-n_0}}{P_p(E')}\right)\\
& \leq & \left(\prod_{p\in\PP} P_p(E')\right) \prod_{p\in\PP}\left(1+\frac{1}{c}p^{-n_0}\right)\\
& \leq & \left(\prod_{p\in\PP} P_p(E')\right)\left(1 + \sum_{p\in\PP}\left(\frac{1}{c}p^{-n_0}\right)\right)
\end{eqnarray*} 

\begin{eqnarray*} 
\phantom{P(E)} & = & P(E')\left(1 + \frac{1}{c}\underbrace{\sum_{p\in\PP}p^{-n_0}}_{\to 0 \text{ for } n_0 \to \infty}\right)\\
& \stackrel{n_0\rightarrow \infty}{\rightarrow} & P(E').
\end{eqnarray*} 

This proves equation \eqref{eq:finiteexhaustion}.\medskip

Now let $A_j\in\DDD$ be a countable family with $D \subseteq \bigcup_{j=1}^{\infty} A_j$. We need to show that $P(D) \leq \sum_{j=1}^{\infty}P(A_j)$.

Recall that $D=D(\EE,\FFF)$. Let $\EE' \leq \EE$ be finite. Then $D(\EE',\FFF) \subseteq D \subseteq  \bigcup_{j=1}^{\infty} A_j$, so by lemma \ref{lem:finite}, we have \vspace{-1ex}

$$P(D(\EE',\FFF)) \leq \sum_{j=1}^{\infty}P(A_j).$$\vspace{-1ex}

Therefore,\vspace{-1ex}

$$P(D) \stackrel{\ref{eq:finiteexhaustion}}{=} \sup_{\EE' \leq \EE \text{ finite}} P(D(\EE',\FFF)) \leq \sum_{j=1}^{\infty}P(A_j),$$

which finishes the proof.
\end{proof}\pagebreak[2]

For the last step, we use

\begin{theorem}[Carath\'{e}odory]
Let $X$ be some space with outer measure $\nu$. We call a set $A\subseteq X$ \emph{measurable}, if for all $B\subseteq X$ we have

$$\nu(B) = \nu(B\setminus A) + \nu(B\cap A).$$

Then the set of all measurable sets is a $\sigma$-algebra, and $\nu$ is a measure when restricted to measurable sets. 
\end{theorem}

\begin{proof}
\cite{Hal50}
\end{proof}

So we only need to show that all uniform properties are measurable (in the sense of Carath\'{e}odory):

\begin{prop}
Let $E$ be a uniform property. Then $O(E)$ is measurable.
\end{prop}

\begin{proof}
Let $A \subseteq \GG$. We need to show that $\nu(A) = \nu(A\setminus O(E)) + \nu(A\cap O(E)).$

Since $\nu$ is subadditive (as outer measure), we only need to show the direction $$\nu(A) \geq \nu(A\setminus O(E)) + \nu(A\cap O(E)).$$

Let $A_i \in \DDD$ be a family such that $A \subseteq \bigcup_{i=1}^{\infty} A_i$. By definition of $\nu$, it suffices to show that for any such family

$$\sum_{i=1}^{\infty} P(A_i) \geq \nu(\underbrace{A\setminus O(E)}_{=:B}) + \nu(\underbrace{A\cap O(E)}_{=: C}).$$

Since $A_i\in \DDD$, we also have $B_i := A_i \setminus O(E) \in \DDD$ and $C_i := A_i \cap O(E) \in \DDD$. Therefore, by proposition \ref{prop:nuP}, we have $\nu(B_i) = P(B_i)$, $\nu(C_i) = P(C_i)$, and $P(A_i) = P(B_i)+P(C_i)$.

Clearly the $B_i$ cover $B$, and the $C_i$ cover $C$, so by definition of $\nu$
\begin{eqnarray*}
\nu(B) & \leq & \sum_{i=1}^{\infty} P(B_i)\text{, and}\\
\nu(C) & \leq & \sum_{i=1}^{\infty} P(C_i).
\end{eqnarray*}

Putting things together, we obtain
\begin{eqnarray*}
\sum_{i=1}^{\infty} P(A_i) & = & \sum_{i=1}^{\infty} (P(B_i) + P(C_i)) \\
& = & \sum_{i=1}^{\infty} P(B_i) + \sum_{i=1}^{\infty} P(C_i) \\
& \geq & \nu(B) + \nu(C),
\end{eqnarray*}

as required.

\end{proof}

So we have successfully concluded the proof and shown that the Cohen-Lenstra measure (def. \ref{def:globalmeasure}) is indeed a well-defined probability measure.

\section{Modifications of the global measure}

There are some important applications of the Cohen-Lenstra heuristic where we need to exclude certain primes. E.g., for quadratic number fields we need to exclude $p=2$. In this case, we proceed as follows: We consider the set $\GG^{\neq 2}$ of all finite abelian groups with trivial $2$-part and modify our definition of uniform properties to these groups. It is clear that all our proofs work also for $\GG^{\neq 2}$ instead of $\GG$, so we get a probability measure on $\GG^{\neq 2}$ that makes all (modified) uniform properties measurable.

Then either we stop at this point and do not make any statements about groups with non-trivial $2$-part. In this case we often replace a random $G$ by $G/G_2$, where $G_2$ is the $2$-part of $G$. Or, if we are given a probability measure on the set $\GG_2$ of all finite abelian $2$-groups, then we take the product space of $\GG^{\neq 2}$ and $\GG_2$ and obtain automatically a probability measure on the product space. Candidates for such probability measures for ``bad'' primes are known for number fields (cf.\! \cite{Mal08} or \cite[sect. 6.1.2]{Len09}).

Of course, all this applies also to other primes than $p=2$, and also to a finite number of primes. 

CAUTION: We get a different probability space for each finite set of primes, and \emph{those probability spaces are not compatible}. As we have seen in section \ref{sect:content}, there is no rich probability measure whose $\sigma$-algebra would make all projections $\GG \to \GG_p$ continuous. 

So there are no objections against ruling out some bad primes in a number field situation (in the sense above), since these primes are fixed. But if you fix one situation and make statements about the $p$-parts of the class groups for various $p$ (as it is often done, e.g.\! in \cite{CL83}), then you must be extremely careful, because our analysis above has shown that you will \emph{inevitably lose countable additivity}. Therefore, the interpretation as probabilities is not valid in this context! Unfortunately, this point is usually ignored in the literature.

A more general way of extending uniform properties is to split up the primes into a finite number of subsets, e.g., into $\PP_1 :=\{p\in\PP \mid p \equiv 1 \bmod 4\}$,  $\PP_2:=\{2\}$, and $\PP_3 :=\{p\in\PP \mid p \equiv 3 \bmod 4\}$. Then we may define uniform properties for each of the sets $\GG_{\PP_1}$, $\GG_{\PP_2}$, and $\GG_{\PP_3}$ (in the obvious way), and by combining them we obtain a probability measure on $\GG$ that is an extension of the probability measure we have defined in the preceding sections. In this way, we may formulate equidistribution statements for congruence classes of primes. However, we have the same restriction as we have when taking out finitely many primes: Each partition of the set $\PP$ yields its own probability measure, and combining more than finitely many of them will eventually result in losing the countable additivity. So statements must still be formulated carefully.\bigskip

Another extension is obvious from measure theory: Of course, we are not restricted to measuring \emph{properties}, but we may measure any measurable function, which includes measuring expected values, higher moments of random variables, and so on. This seems like a trivial remark, but so far it has been an unsolved problem which quantities to consider in the Cohen-Lenstra context. Cohen and Lenstra declared that we should take ``reasonable'' functions without specifying what ``reasonable'' means, and this handwaving concept was adapted in basically all subsequent papers. By our preparatory work, we get the solution for this problem for free from measure theory.

For convenience, let me explicitly state what it means for a sequence of groups to be random (more precisely: equidistributed) with respect to the Cohen-Lenstra measure:

\begin{definition}\label{def:randomglobal}
Let $(G_i)_{i=1}^{\infty}$ be a sequence of finite abelian groups. Let $\Sigma$ be the $\sigma$-algebra on $\GG$ generated by uniform properties and let $\mu$ be the probability measure on $\Sigma$ as defined in \ref{def:globalmeasure}. We say that $G_i$ \emph{behaves as a random sequence} or \emph{is equidistributed with respect to the Cohen-Lenstra measure}\index{equidistributed sequence} if for all measurable functions $f:\GG\rightarrow \CC$ we have

$$\lim_{n\to \infty}\frac{\sum_{i=1}^n f(G_i)}{n} = \int_{\GG}fd\mu.$$ 
\end{definition}